\documentclass[11pt]{amsart}
\usepackage{color}
\usepackage{amsmath}
\usepackage{amssymb}
\usepackage{graphicx}
\newtheorem{theorem}{Theorem}[section]
\newtheorem{corollary}[theorem]{Corollary}
\newtheorem{lemma}[theorem]{Lemma}
\newtheorem{proposition}[theorem]{Proposition}
\theoremstyle{definition}

\newtheorem{remark}[theorem]{Remark}

\numberwithin{equation}{section}

\title[Some results on semiclassical spectral analysis]{Some results on semiclassical spectral analysis of   magnetic Schr\"odinger operators}

\author[Y. A. Kordyukov]{Yuri A. Kordyukov}
\address{Institute of Mathematics, Ufa Federal Research Centre, Russian Academy of Sciences, 112~Chernyshevsky str., 450008 Ufa, Russia} 
\email{{\tt yurikor@matem.anrb.ru}}



\begin{document}
\begin{abstract}
In our recent papers, we studied semiclassical spectral problems for the Bochner-Schr\"odinger operator on a manifold of bounded geometry. We survey some results of these papers in the setting of the magnetic Schr\"odinger operator in the Euclidean space and describe some ideas of the proofs.   
\end{abstract}

\dedicatory{Dedicated to Bernard Helffer on the occasion of his 75th birthday}

 \maketitle
%\tableofcontents
\section{Preliminaries and main results}

In our recent papers \cite{higherLL,jst,UMN-trace,bochner-trace,essential}, we studied semiclassical spectral problems for the Bochner-Schr\"odinger operator on a manifold of bounded geometry. We survey some results of these papers in the setting of the magnetic Schr\"odinger operator in the Euclidean space and describe some ideas of the proofs.   

\subsection{The setting}
We consider a semiclassical Schr\"odinger operator in the Euclidean space $\mathbb R^d$ of the form 
\begin{equation}\label{e:def-Hh}
H_\hbar=\sum_{j=1}^{d}\left(\frac{\hbar}{i}\frac{\partial}{\partial x_j}-A_j(x)\right)^2 +\hbar V(x), \quad \hbar>0,
\end{equation}
where the components of the magnetic potential $A_j, j=1,\ldots,d$, and the electric potential $V$ are real-valued smooth functions in $\mathbb R^d$.

The magnetic field is given by the antisymmetric matrix $B=(B_{jk})$ given by
\[
B_{jk}=\frac{\partial A_k}{\partial x_j}-\frac{\partial A_j}{\partial x_k}, \quad j,k=1,\ldots,d.
\]
We assume that $B_{jk}$ and $V$ are uniformly bounded along with all derivatives of arbitrary order, $B_{jk}, V \in C^\infty_b(\mathbb R^d).$

The operator $H_\hbar$ is essentially self-adjoint in the Hilbert space $L^2(\mathbb R^d)$ with initial domain  $C^\infty_c(\mathbb R^d)$, and we still denote by $H_\hbar$ its unique self-adjoint extension.  

\begin{remark}
In \cite{higherLL,jst,UMN-trace,bochner-trace,essential}, the setting was following. Let $(X,g)$ be a complete Riemannian manifold of dimension $d$, $(L,h^L)$ a Hermitian line bundle on $X$ with a Hermitian connection $\nabla^L$ and $(E,h^E)$ a Hermitian vector bundle of rank $r$ on $X$ with a Hermitian connection $\nabla^E$. Suppose that $(X, g)$ is  a manifold of bounded geometry and $L$ and $E$ have bounded geometry. This means that the curvatures $R^{TX}$, $R^L$ and $R^E$ of the Levi-Civita connection $\nabla^{TX}$, connections $\nabla^L$ and $\nabla^E$, respectively, and their derivatives of any order are uniformly bounded on $X$ in the norm induced by $g$, $h^L$ and $h^E$, and the injectivity radius $r_X$ of $(X, g)$ is positive.

The magnetic field is the real-valued closed differential 2-form $\mathbf B$ given by 
\begin{equation}\label{e:def-omega}
\mathbf B=iR^L. 
\end{equation} 
where $R^L$ is the curvature of the connection $\nabla^L $ defined as $R^L=(\nabla^L)^2$. 

For any $p\in {\mathbb N}$, let $L^p:=L^{\otimes p}$ be the $p$th tensor power of $L$ and let
\[
\nabla^{L^p\otimes E}: {C}^\infty(X,L^p\otimes E)\to
{C}^\infty(X, T^*X \otimes L^p\otimes E)
\] 
be the Hermitian connection on $L^p\otimes E$ induced by $\nabla^{L}$ and $\nabla^E$. Consider the induced Bochner Laplacian $\Delta^{L^p\otimes E}$ acting on $C^\infty(X,L^p\otimes E)$ by
\[
\Delta^{L^p\otimes E}=\big(\nabla^{L^p\otimes E}\big)^{\!*}\,
\nabla^{L^p\otimes E},
\] 
where $\big(\nabla^{L^p\otimes E}\big)^{\!*}: {C}^\infty(X,T^*X\otimes L^p\otimes E)\to
{C}^\infty(X,L^p\otimes E)$ is the formal adjoint of  $\nabla^{L^p\otimes E}$. Let $V\in C^\infty(X,\operatorname{End}(E))$ be a self-adjoint endomorphism of $E$. We assume that $V$ and its derivatives of any order are uniformly bounded on $X$ in the norm induced by $g$ and $h^E$. 
The main object is the Bochner-Schr\"odinger operator $H_p$ acting on $C^\infty(X,L^p\otimes E)$ by
\[
H_{p}=\frac 1p\Delta^{L^p\otimes E}+V. 
\] 
 
The current setting described above is related with this general one as follows.
The Riemannian manifold $(X,g)$ is the Euclidean space $\mathbb R^d$ equipped with the standard flat metric. The Hermitian line bundle $(L,h^L)$ is trivial, and $(E,h^E)$ is the trivial Hermitian line bundle with the trivial connection $\nabla^E$. The Hermitian connection $\nabla^L$ is given by 
$$
\nabla^L=d-i \mathbf A, 
$$ 
where $\mathbf A$ is the real-valued differential 1-form on $\mathbb R^d$ given by
\[
\mathbf A=\sum_{j=1}^d A_j(x)dx_j.
\]
Then the magnetic field $\mathbf B$ coincides with the de Rham differential of $\mathbf A$:
\[
\mathbf B=d\mathbf A=\sum_{j<k} B_{jk}(x)dx_j\wedge dx_k. 
\]
The operator $H_p$ is related with the semiclassical magnetic Schr\"odinger operator $H_\hbar$ given by \eqref{e:def-Hh} as follows:
\[
H_p=\hbar^{-1}H_\hbar, \quad \hbar=\frac{1}{p}. 
\]
In the general setting, the parameter $p$ is a natural number, that gives rise to a discrete set of values of $\hbar$. When the line bundle $L$ is trivial, $p$ can be considered to be an arbitrary non-negative real number. 
\end{remark}

\subsection{Rough asymptotic description of the spectrum}
We will need some particular second order differential operators  (the model operators) associated with an arbitrary point $x_0\in {\mathbb R}^{d}$, introduced by Demailly \cite{Demailly85,Demailly91}. They are obtained from $H_{\hbar}$ by freezing its coefficients at $x_0$.

Let $x_0\in \mathbb R^d$.  
Consider the magnetic potential 
\begin{equation}\label{e:Ajx0}
A_{j,x_0}(Z)=\frac{1}{2}\sum_{k=1}^dB_{k j}(x_0)\,Z_k, \quad Z\in \mathbb R^d, \quad j=1,\ldots,d, 
\end{equation}
with constant magnetic field $B_{x_0}=(B_{jk}(x_0))_{j,k=1,\ldots,d}$.

The model operator $\mathcal H^{(x_0)}_\hbar$ is the semiclassical Schr\"odinger operator in $\mathbb R^d$ given by
\[
\mathcal H^{(x_0)}_\hbar=\sum_{j=1}^{d}\left(\frac{\hbar}{i}\frac{\partial}{\partial Z_j}-A_{j,x_0}(Z)\right)^2 +\hbar V(x_0),
\]
Here it is natural to consider $Z$ as a tangent vector to $\mathbb R^d$ at $x_0$, so the model operator is a second order differential operator on the tangent space $T_{x_0}\mathbb R^d\cong \mathbb R^d$. 
By simple rescaling, this operator is unitarily equivalent to the operator $\hbar\mathcal H^{(x_0)}$, where
\begin{equation}\label{e:DeltaL0p}
\mathcal H^{(x_0)}=\mathcal H^{(x_0)}_1=\sum_{j=1}^{d}\left(\frac{1}{i}\frac{\partial}{\partial Z_j}-A_{j,x_0}(Z)\right)^2 +V(x_0).
\end{equation}

Suppose that the rank of $B_{x_0}$ equals $2n=2n_{x_0}$. Its non-zero eigenvalues have the form $\pm i a_j({x_0}), j=1,\ldots,n,$ with $a_j({x_0})>0$ and, if $d>2n$, zero is an eigenvalue of multiplicity $d-2n$.   

For $\mathbf k=(k_1,\cdots,k_n) \in \mathbb Z^n_+$, set
\begin{equation}\label{e:def-Lambda}
\Lambda_{\mathbf k}(x_0)=\sum_{j=1}^n(2k_j+1) a_j(x_0)+V(x_0).
\end{equation}

In the full-rank case $d=2n$, the spectrum of $\mathcal H^{(x_0)}$ is a countable set of eigenvalues of infinite multiplicity:
\[
\sigma(\mathcal H^{(x_0)})=\Sigma_{x_0}:=\left\{\Lambda_{\mathbf k}({x_0})\,:\, \mathbf k\in{\mathbb Z}_+^n\right\}. 
\]
If $d>2n$, the spectrum of $\mathcal H^{(x_0)}$ is a semiaxis:
\[
\sigma(\mathcal H^{(x_0)})=[\Lambda_0(x_0), +\infty),
\]
where 
\[
\Lambda_0(x_0):=\sum_{j=1}^n a_j(x_0)+V(x_0). 
\]

\begin{theorem}[\cite{higherLL}]\label{t:spectrum}
Assume that the magnetic field $B$ is of full rank at each point $x_0\in {\mathbb R}^{2n}$. Then, for any $K>0$, there exist $c>0$ and $\hbar_0>0$ such that, for any $\hbar\in (0,\hbar_0]$, the spectrum of $H_\hbar$ in the interval  $[0,K\hbar]$  is  contained in the $c\hbar^{5/4}$-neighborhood of $\hbar\Sigma$, where 
$\Sigma$ be the union of the spectra of the model operators: 
\begin{equation}\label{e:def-Sigma}
\Sigma=\bigcup_{x_0\in {\mathbb R}^{2n}}\Sigma_{x_0}=\left\{\Lambda_\mathbf {k}(x_0)\,:\, \mathbf k\in{\mathbb Z}_+^n, x_0\in {\mathbb R}^{2n} \right\}.
\end{equation}  
\end{theorem}

In \cite{charles21}, L. Charles proved a better asymptotic estimate for the spectrum of the Bochner Laplacian on a compact manifold than in \cite{higherLL}. It is quite possible that the technique developed in \cite{charles21} allows us to prove an optimal estimate of order $\hbar^{3/2}$ instead of $\hbar^{5/4}$ as in Theorem~\ref{t:spectrum} in the current setting. 

When the magnetic field is non-degenerate and has discrete wells, an asymptotic description of the spectrum of $H_\hbar$ in an interval of the form $[0,K\hbar]$ in terms of the spectrum of an effective pseudodifferential operator is given in dimension two by Helffer-Kordyukov \cite{HK15} and Raymond-V\~{u} Ng\d{o}c \cite{RV} (see also a survey paper \cite{HK14}) and in higher dimensions by Morin \cite{M}. For constant rank magnetic fields with discrete wells, similar results are obtained by Helffer-Kordyukov-Raymond-V\~{u} Ng\d{o}c \cite{HKRV16} in dimension 3 and by Morin \cite{M24} in higher dimensions.

\subsection{Functions of the Schr\"odinger operator}
For any $\varphi\in \mathcal S(\mathbb R)$, the linear bounded operator $\varphi(H_{\hbar}/\hbar)$ in $L^2(\mathbb R^d)$ is defined by the spectral theorem. It is a smoothing operator with smooth Schwartz kernel $K_{\varphi(H_{\hbar}/\hbar)}\in C^\infty(\mathbb R^d\times \mathbb R^d)$:
\[
\varphi(H_{\hbar}/\hbar)u(x) =\int_{\mathbb R^d} K_{\varphi(H_{\hbar}/\hbar)}(x,x^\prime)u(x^\prime)dx^\prime, \quad u\in L^2(\mathbb R^d). 
\]
Using the finite propagation speed property of solutions of hyperbolic equations, one can show that, for any $\varepsilon >0$ and for any multi-indices $\alpha,\alpha^\prime\in \mathbb Z^d_+$,
\begin{equation}\label{e:off-diag}
\left|\partial^{\alpha}_x\partial^{\alpha^\prime}_{x^\prime} K_{\varphi(H_{\hbar}/\hbar)}(x,x^\prime)\right|=\mathcal O(\hbar^{\infty}), \quad |x-x^\prime|>\varepsilon.
\end{equation}

The following theorem states a complete asymptotic expansion for the function $K_{\varphi(H_{\hbar}/\hbar)}$ as $\hbar\to 0$ in a fixed neighborhood of the diagonal (independent of $\hbar$).

\begin{theorem}[\cite{bochner-trace}]\label{t:main}
There are a family of smooth real-valued functions $\Phi^{(x_0)}\in C^\infty(\mathbb R^d)$ parameterized by $x_0\in \mathbb R^d$ and a sequence of smooth functions $F_{r,x_0}(Z, Z^\prime)$, $r\geq 0$, of variables $x_0\in \mathbb R^d$ and $Z,Z^\prime\in T_{x_0}\mathbb R^d\cong \mathbb R^d$, such that the following asymptotic expansion holds true uniformly on $x_0\in \mathbb R^d$:
\begin{multline}\label{e:main}
\hbar^{\frac d2} K_{\varphi(H_{\hbar}/\hbar)}(x_0+Z, x_0+Z^\prime)\\ \cong e^{i(\Phi^{(x_0)}(x_0+Z)-\Phi^{(x_0)}(x_0+Z^\prime))/\hbar} \sum_{r=0}^\infty F_{r,x_0}(\hbar^{-\frac 12} Z, \hbar^{-\frac 12}Z^\prime)\hbar^{\frac{r}{2}}, \quad \hbar\to 0.
\end{multline}
More precisely, for any $j\in \mathbb Z_+$, the function
\begin{multline*}
R_{j,\hbar,x_0}(Z,Z^\prime)=\hbar^{\frac d2} K_{\varphi(H_{\hbar}/\hbar)}(x_0+Z, x_0+Z^\prime)\\ -e^{i(\Phi^{(x_0)}(x_0+Z)-\Phi^{(x_0)}(x_0+Z^\prime))/\hbar} \sum_{r=0}^jF_{r,x_0}(\hbar^{-\frac 12} Z, \hbar^{-\frac 12}Z^\prime)\hbar^{\frac{r}{2}}
\end{multline*}
satisfies the following condition. For any $m\in \mathbb Z_+$, there exists $M\in {\mathbb N}$ such that, for any $N\in {\mathbb N}$, there exists $C>0$ such that for any $\alpha, \alpha^\prime \in \mathbb Z_+^d$ with $|\alpha|+|\alpha^\prime|\leq m$, we have
\begin{multline*}
\left|\frac{\partial^{|\alpha|+|\alpha^\prime|}R_{j,\hbar,x_0}}{\partial Z^\alpha\partial Z^{\prime\alpha^\prime}}(Z,Z^\prime)\right| \\
\leq C\hbar^{-\frac{j-m+1}{2}}(1+\hbar^{-\frac 12}|Z|+\hbar^{-\frac 12}|Z^\prime|)^{M} (1+\hbar^{-\frac 12}|Z-Z^\prime|)^{-N}, \\
\hbar\geq 0,\quad x_0\in \mathbb R^d,\quad Z,Z^\prime\in \mathbb R^d.
\end{multline*}
\end{theorem}

Putting $Z=Z^\prime=0$, we get on-diagonal expansion for $K_{\varphi(H_{\hbar}/\hbar)}$. 

\begin{corollary}\label{c:main} 
For any $x_0\in \mathbb R^d$, there exists a sequence of distributions $f_r(x_0) \in \mathcal S^\prime(\mathbb R), r\geq 0$,  such that the following asymptotic expansion holds true as $\hbar \to 0$  uniformly on $x_0$:
\begin{equation}\label{e:cmain}
K_{\varphi(H_{\hbar}/\hbar)} (x_0,x_0)\sim \hbar^{-\frac d2}\sum_{r=0}^\infty \langle f_{r}(x_0), \varphi\rangle \hbar^{\frac{r}{2}}, \quad \langle f_{r}(x_0), \varphi\rangle=F_{r,x_0}(0,0).
\end{equation}
\end{corollary}

On a compact manifold, Corollary \ref{c:main} immediately implies an asymptotic expansion for the trace of the operator $\varphi(H_{\hbar}/\hbar)$ with $\varphi\in C^\infty_c(\mathbb R)$ of the form
\begin{equation}\label{e:trace}
\operatorname{tr} \varphi(H_{\hbar}/\hbar)\sim \hbar^{-\frac d2}\sum_{r=0}^{\infty}\langle f_{r}, \varphi\rangle \hbar^{\frac{r}{2}}, \quad \hbar \to 0.
\end{equation}
The formula \eqref{e:trace} is a particular case of the Gutzwiller trace formula for the semiclassical magnetic Schr\"odinger operator and the zero energy level of the classical Hamiltonian, which is critical in this case. We refer the interested reader to the survey \cite{UMN-trace} for more information. 

In the current setting of the Euclidean space $\mathbb R^d$, the spectrum of $H_\hbar$ is, generally, continuous, so the operator $\varphi(H_{\hbar}/\hbar)$ is not of trace class. Nevertheless, Theorem~\ref{t:ess-spectrum} below provides a sufficient condition for the spectrum of $H_\hbar$ to be discrete on some interval $(\hbar\alpha,\hbar\beta)$. Combined with a rough polynomial bound for the number of eigenvalues in $(\hbar\alpha,\hbar\beta)$ and the exponential localization result for the corresponding eigenfunctions given by Theorem~\ref{t:eigenest}, this allows us to show that the operator $\varphi(H_{\hbar}/\hbar)$ with $\varphi\in C^\infty_c(\mathbb R)$ supported in the interval $(\hbar\alpha,\hbar\beta)$ is of trace class for sufficiently small $\hbar>0$ and, using Theorem~\ref{t:main}, we can prove that its trace admits an asymptotic expansion as in \eqref{e:trace}. The details will be given elsewhere. 

Explicit formulas for the coefficients $F_{r,x_0}$ and $f_{r}(x_0)$ in the asymptotic expansions \eqref{e:main} and \eqref{e:cmain} involve the model operators introduced above. First of all, we have the following theorem for the leading coefficients. 

\begin{theorem}\label{t:leading-coefficient}
The leading coefficient $F_{0,x_0}$ is the Schwartz kernel $K_{\varphi(\mathcal H^{(x_0)})}$ of the corresponding function $\varphi(\mathcal H^{(x_0)})$ of the model operator $\mathcal H^{(x_0)}$:
\[
F_{0,x_0}(Z,Z^\prime)=K_{\varphi(\mathcal H^{(x_0)})}(Z,Z^\prime),  \quad Z, Z^\prime\in \mathbb R^d.
\]
As a consequence, we get 
\[
f_{0}(x_0)=K_{\varphi(\mathcal H^{(x_0)})}(0,0).
\]
\end{theorem}

One can compute the Schwartz kernel $K_{\varphi(\mathcal H^{(x_0)})}$ and get an explicit formula for $f_{0}(x_0)$. In the full-rank case $d=2n$, we have
\[
f_{0}(x_0)=\frac{1}{(2\pi)^{n}} \left(\prod_{j=1}^n a_j(x_0)\right)  \sum_{\mathbf k\in{\mathbb Z}_+^n}\varphi(\Lambda_{\mathbf k}(x_0)), 
\]
and, for $d>2n$,
\[
f_{0}(x_0)=\frac{1}{(2\pi)^{d-n}} \left(\prod_{j=1}^n a_j(x_0)\right)  \sum_{\mathbf k\in{\mathbb Z}_+^n}\int_{{\mathbb R}^{d-2n}} \varphi(|\xi|^2+\Lambda_{\mathbf k}(x_0))d\xi.
\]

For the next coefficients  in the asymptotic expansion \eqref{e:main}, we get 
\begin{equation}\label{e:Fr}
F_{r,x_0}(Z,Z^\prime)=\sum_{|\mathbf k|\leq 3r} \sum_{\ell=1}^{N_r} D_{\mathbf k,\ell} K_{\varphi^{(\ell-1)}(\mathcal H^{(x_0)}-2\mathbf k\cdot a)}(Z,Z^\prime),  \quad Z, Z^\prime\in \mathbb R^d,
\end{equation}
where $D_{\mathbf k,\ell}$ is a differential operator with polynomial coefficients in $Z$ of order $3r$ and $N_r$ is a some natural number. 

Assume that the magnetic field $B$ is of full rank $d=2n$ at each point $x_0\in {\mathbb R}^d={\mathbb R}^{2n}$ and $\inf |B_{x_0}|\geq c>0$ for any $x_0\in {\mathbb R}^{2n}$. Consider the spectral projection $E_{[\hbar a,\hbar b]}$ of the operator $H_\hbar$ associated to an interval $[\hbar a,\hbar b]$ such that $a, b\not \in \Sigma$. By Theorem \ref{t:spectrum}, there exists $\mu_0>0$ and $\hbar_0>0$ such that for any $\hbar\in (0,\hbar_0)$ the extreme points $\hbar a$ and $\hbar b$ of the interval are in gaps of the spectrum of the operator $H_\hbar$. So, replacing $E_{[\hbar a,\hbar b]}$ with the operator $\varphi(H_\hbar/\hbar)$ for an appropriate function $\varphi\in C^\infty(\mathbb R^{2n})$, we  infer that the off-diagonal estimate \eqref{e:off-diag} and the asymptotic expansion of Theorem~\ref{t:main} hold true for the smooth Schwartz kernel $E_{[\hbar a,\hbar b]}(x,x^\prime)$, $x,x^\prime\in \mathbb R^{2n}$. 

In fact, better estimates can be proved \cite{higherLL}.   
First, one can show an exponential off-diagonal estimate for $E_{[\hbar a,\hbar b]}(x,x^\prime)$, which is the analog of \cite[Theorem 1]{ma-ma15} and \cite[Theorem 1.2]{ko-ma-ma} for the Bergman kernel. 
There exists $c>0$ such that for any $k\in \mathbb N$, there exists $C_k>0$ such that for any multi-indices $\alpha$ and $\alpha^\prime$ with $|\alpha|+|\alpha^\prime|\leq k$, for any $\hbar\in (0,\hbar_0)$, $x, x^\prime \in \mathbb R^{2n}$, we have
\[
\big| \partial^{\alpha}_x\partial^{\alpha^\prime}_{x^\prime} E_{[\hbar a,\hbar b]}(x, x^\prime)\big|\leq C_k \hbar^{-(n+\frac{k}{2})} e^{-c\,|x-x^\prime|/\hbar^{1/2}}.
\]
Second, one can prove a similar exponential estimate for the remainder $R_{j,\hbar,x_0}$ in the asymptotic expansion of Theorem~\ref{t:main} for $E_{[\hbar a,\hbar b]}(x,x^\prime)$ (see \cite{higherLL} for more details). We also mention that, on a compact manifold, an asymptotic expansion of $E_{[\hbar a,\hbar b]}(x,x^\prime)$ (without exponential estimates) was proved by Charles in \cite{charles21}.

For a compact manifold, the Toeplitz operator calculus associated with the spectral projection $E_{[\hbar a,\hbar b]}$ as above is developed in \cite{charles24,jst}. It is shown that this calculus provides a Berezin-Toeplitz quantization of the manifold equipped with the symplectic form $\mathbf B$. In \cite{charles24}, Charles also computes the dimension of the corresponding eigenspace as a Riemann-Roch number. The constructions of \cite{jst} are extended  in \cite{bg-guant} to manifolds of bounded geometry, in particular, to the current setting of the Euclidean space.

\subsection{Localization of the spectral projection and eigenfunctions}
As an immediate consequence of Corollary \ref{c:main} and the description of the coefficients of the asymptotic expansion given by \eqref{e:Fr}, we obtain the following result on the asymptotic  localization of the Schwartz kernel of the spectral projection on the diagonal in the case when the magnetic field is of full rank, which is an improvement of \cite[Theorem 1.3]{charles21}.

\begin{theorem}[\cite{bochner-trace}]\label{t:local}
Assume that, for some $x_0\in \mathbb R^d$, the rank of $B_{x_0}$ equals $d=2n$ and an interval $[a,b]$ does not contain any $\Lambda_{\mathbf k}(x_0)$ with  $\mathbf k \in \mathbb Z^n_+$. 

For any $\varphi\in \mathcal S({\mathbb R})$ such that ${\rm supp}\,\varphi \subset (a,b)$, 
\[
\left|K_{\varphi(H_{\hbar}/\hbar)}(x_0,x_0)\right|_{{C}^k}=\mathcal O(\hbar^{\infty}), \quad k=0,1,\ldots, \quad \hbar\to 0.
\]
Moreover, if an interval $[a,b]$ does not contain any $\Lambda_{\mathbf k}(x_0)$ with  $\mathbf k \in \mathbb Z^n_+$, then the Schwartz kernel of the spectral projection $E_{[\hbar a,\hbar b]}$ of the operator $H_\hbar$ associated to $[\hbar a,\hbar b]$ satisfies 
\[
\left|E_{[\hbar a,\hbar b]}(x_0,x_0)\right|=\mathcal O(\hbar^{\infty}),\quad \hbar\to 0.
\]
\end{theorem}

By Theorem~\ref{t:local}, we conclude that, if, for some $x_0\in \mathbb R^{2n}$, an interval $[a,b]$ does not contains any $\Lambda_{\mathbf k}(x_0)$ with $\mathbf k \in \mathbb Z^n_+$, then, for any sequence $\{u_{\hbar}\in C^\infty(\mathbb R^{2n})\cap L^2({\mathbb R}^{2n}), \hbar>0\}$ of eigenfunctions of $H_\hbar$ with the corresponding eigenvalues $\lambda_{\hbar}$ in $[\hbar a,\hbar b]$, we have
\[
 |u_\hbar(x_0)|=\mathcal O(\hbar^{\infty}), \quad \hbar\to 0.
\]
In other words, the essential support of the sequence $\{u_\hbar, \hbar>0\}$ is contained in the closed subset $\mathcal K_{[a,b]}\subset \mathbb R^{2n}$ of all $x_0\in \mathbb R^d$ such that $\Lambda_{\mathbf k}(x_0)\in [a,b]$ for some $\mathbf k \in \mathbb Z^n_+$: 
\[
\mathcal K_{[a,b]}=\{x\in \mathbb R^{2n} \,:\,  \Sigma_x\cap [a,b]\neq\emptyset\}.
\]
This result complements the asymptotic description of the spectrum of $H_\hbar$ in terms of $\Lambda_{\mathbf k}(x_0)$ given in \cite{charles21,higherLL}.

Theorem~\ref{t:local} states that the spectral function $E_{[\hbar a,\hbar b]}(x_0,x_0)$ outside of $\mathcal K_{[a,b]}$ has rapid decay in $\hbar$ as $\hbar\to 0$. It gives no information on dependence of this decay on the distance from $x_0$ to $\mathcal K_{[a,b]}$. Such an information is obtained in \cite{essential} in the case when the magnetic field is of full rank and $\mathcal K_{[a,b]}$ is compact. Under this assumption, we show that this part of spectrum is discrete and the corresponding eigenfunctions exponentially localized away the compact set $\mathcal K_{[a,b]}$. In particular, each eigenfunction outside of $\mathcal K_{[a,b]}$ has pointwise exponential decay in $\hbar$  as $\hbar\to 0$.
 
\begin{theorem}[\cite{essential}]\label{t:ess-spectrum}
Assume that the magnetic field $B$ is of full rank at each point $x_0\in {\mathbb R}^{2n}$ and $\inf |B_{x_0}|\geq c>0$ for any $x_0\in {\mathbb R}^{2n}$. Moreover, assume that, for $[a,b]\subset \mathbb R$, the set $\mathcal K_{[a,b]}$ is compact. Then there exist $\epsilon>0$ and $\hbar_0>0$ such that for any $\hbar\in (0,\hbar_0]$ the spectrum of $H_{\hbar}$ in $[\hbar a+\epsilon \hbar^{5/4}, \hbar b-\epsilon \hbar^{5/4}]$ is discrete.  
\end{theorem}

As in Theorem~\ref{t:spectrum}, the order $\hbar^{5/4}$ doesn't seem to be optimal and, probably, can be improved.

\begin{theorem}[\cite{essential}]\label{t:eigenest}
Under the assumptions of Theorem \ref{t:ess-spectrum}, for any $a_1>a$ and $b_1<b$, there exist $\hbar_0>0$ and $C, c>0$ such that, for any eigenfunction $u_{\hbar}\in C^\infty({\mathbb R}^{2n})\cap L^2({\mathbb R}^{2n})$ of the operator $H_{\hbar}$ with $\hbar\in (0,\hbar_0]$ and with the corresponding eigenvalue $\lambda_{\hbar}\in [\hbar a_1, \hbar b_1]$:
\[
H_{\hbar}u_{\hbar}=\lambda_{\hbar}u_{\hbar},
\]
we have 
\[
\int_{{\mathbb R}^{2n}} e^{2cd (x,\mathcal K_{[a,b]})/{\hbar}^{1/2}}|u_{\hbar}(x)|^2dx \leq C\|u_{\hbar}\|^2,
\]
where $d (x,\mathcal K_{[a,b]})$ is the distance from $x$ to $\mathcal K_{[a,b]}$. 
\end{theorem}

The exponential decay in Theorem \ref{t:eigenest} is probably non-optimal in some cases. For instance, in the case of two-dimensional magnetic wells, a faster exponential decay of low energy eigenfunctions was proven in \cite{BRV} under analyticity assumptions on the magnetic field. It seems to be optimal in the $C^\infty$-category, see also a discussion on p. 57 in \cite{BRV}.

This study is partly motivated by the spectral theory of Schr\"odinger operators with magnetic walls (see discussion and some references in \cite{essential}). It is highly interesting to study the case when the set $\mathcal K_{[a,b]}$ is non-compact.   In this case, the spectrum of $H_{\hbar}$ in $(\hbar a,\hbar b)$ is, generally, continuous, but one can consider functions in the image of the spectral projection $E_{[\hbar a,\hbar b]}$ of $H_{\hbar}$ associated to $(\hbar a,\hbar b)$. In the quantum Hall effect theory, they are usually called edge states. For the Iwatsuka model proposed in \cite{I85} and its generalizations, exponential localization of edge states was established in \cite{HS15} (see also \cite{RS23} and references therein). In this model, $\mathcal K_{[a,b]}$ is a strip in the plane.

\section{Functions of the Schr\"odinger operator}

Theorem~\ref{t:main} is the analog of Theorem 4.18' in \cite{dai-liu-ma} and Theorem 1 in \cite{Kor18} on the full off-diagonal expansions for the (generalized) Bergman kernels. In the case of Bergman kernels, the remainder estimates in asymptotic expansions are exponentially decreasing for large $|Z-Z^\prime|$. It is clear that, for general functions from the Schwartz class,  these remainder estimates should be only rapidly decreasing. 

The proof combines methods of functional analysis (first of all, the functional calculus based on the Helffer-Sj\"{o}strand formula \cite{HS-LNP345} and norm estimates in suitable Sobolev spaces) with  methods of local index theory developed in \cite{dai-liu-ma,ko-ma-ma,ma-ma:book,ma-ma08,ma-ma15} for the study of the asymptotic behavior of the (generalized) Bergman kernels and originating in the paper by Bismut and Lebeau \cite{BL}. Note that, in contrast to the above mentioned papers \cite{dai-liu-ma,ma-ma:book,ma-ma08}, we do not require that the magnetic field be of full rank. A similar strategy was applied in a close situation by  Savale in \cite{Savale17,Savale18}. 

\subsection{A particular gauge}\label{gauge} 
We start with choosing a particular magnetic potential adapted to an arbitrary point $x_0\in \mathbb R^d$. First, we introduce new coordinates $Z\in \mathbb R^d$ by the formula $x=x_0+Z$. As mentioned above, it is convenient to think of $Z$ as a tangent vector to $\mathbb R^d$ at $x_0$. 

Let us consider the radial vector field $\mathcal R$ in $\mathbb R^d$ given by
\[
\mathcal R=\sum_{j=1}^dZ_j\frac{\partial}{\partial Z_j}. 
\]
Given the magnetic potential $\mathbf A=\sum_{j=1}^d A_j(x_0+Z)dZ_j$, we make a gauge transformation  such that the new magnetic potential 
$$\mathbf A^{(x_0)}=\mathbf A+d\Phi^{(x_0)}=\sum_{j=1}^dA^{(x_0)}_j(x_0+Z)dZ_j,$$ 
satisfies the conditions  
\[
\mathbf A^{(x_0)}(x_0)=0, \quad 
\iota_{\mathcal R}\mathbf A^{(x_0)}(x_0+Z)=\sum_{j=1}^dZ_jA^{(x_0)}_j(x_0+Z)=0.
\]

In a geometric language, we use the trivialization of the trivial line bundle over $\mathbb R^d$ by parallel transport along the rays $t\in \mathbb R_+\mapsto x_0+tZ\in \mathbb R^d$. This gauge is sometimes called the Fock-Schwinger, Poincar\'e or transverse gauge. In local index theory, it was introduced in \cite{ABP73} and called the synchronous framing. Such a gauge was used by Helffer and Mohamed in \cite{HM88}. 
It plays a crucial role in the gauge covariant magnetic perturbation theory elaborated by H. Cornean and G. Nenciu (see, for instance, \cite{CN98,CN00,N02}).

In order to find such a gauge transformation, we have to solve the differential equation 
\[
\mathcal R[\Phi^{(x_0)}](x_0+Z)=-\iota_{\mathcal R}\mathbf A(x_0+Z),
\]
which gives
\[
\Phi^{(x_0)}(x_0+Z)=-\sum_{j=1}^d\int_0^1 A_j(x_0+\tau Z)Z_j\,d\tau.
\]
The potential $\mathbf A^{(x_0)}$ is given by
\[
A^{(x_0)}_j(x_0+Z)=\sum_{k=1}^d\left(\int_0^1 B_{kj}(x_0+\tau Z)\tau\,d\tau\right) Z_k, \quad j=1,\ldots,d.
\]

Expanding both parts of the last formula in Taylor series in $Z$, we get the following well-known fact proved in \cite[Appendix II]{ABP73} (see also \cite[Lemma 1.2.4]{ma-ma:book}): for any $r\geq 0$ and $j=1,\ldots,d$, 
\begin{equation}\label{e:ABP73}
\sum_{|\alpha|=r}\partial^\alpha A^{(x_0)}_j(x_0)\frac{Z^\alpha}{\alpha!}=\frac{1}{r+1}\sum_{|\alpha|=r-1}\sum_{k=1}^d\partial^\alpha B_{kj}(x_0)\,Z_k\frac{Z^\alpha}{\alpha!}.
\end{equation}
In particular, we have
\begin{equation}\label{e:ABP73a}
A^{(x_0)}_j(x_0+Z)=A_{j,x_0}(Z)+\mathcal O(|Z|^2), \quad j=1,\ldots,d,
\end{equation}
 where $A_{j,x_0}(Z)$ is given by \eqref{e:Ajx0}.

Finally, we have the corresponding unitary transformation of operators:
\begin{equation}\label{e:H-equiv}
 H_\hbar=e^{i\Phi^{(x_0)}/\hbar} H^{(x_0)}_\hbar e^{-i\Phi^{(x_0)}/\hbar},
\end{equation}
where
\[
H^{(x_0)}_\hbar=\sum_{j=1}^{d}\left(\frac{\hbar}{i}\frac{\partial}{\partial x_j}-A^{(x_0)}_j(x)\right)^2 +\hbar V(x), \quad \hbar>0.
\]

\subsection{Rescaling and formal expansions}\label{scale}
We use the rescaling introduced in \cite[Section 1.2]{ma-ma08}. 

Denote $t=\sqrt{\hbar}$. For $u\in C^\infty(\mathbb R^d)$, set
\[
S_t u(x)=u\left(\frac{x-x_0}{t}\right), \quad x\in \mathbb R^d.
\]
The operator $S_t$ depends on $x_0\in \mathbb R^d$, but, for simplicity of notation, we will often omit $x_0$.

Consider the rescaling of the operator $H^{(x_0)}_\hbar/\hbar$ defined by 
\begin{equation}\label{scaling}
\mathcal H_t=S^{-1}_t (H^{(x_0)}_\hbar/\hbar) S_t.
\end{equation}
By construction, it is a self-adjoint operator in $L^2(\mathbb R^d)$, and its spectrum coincides with the spectrum of $H_\hbar/\hbar$. 

We get a family $\mathcal H_t$ of second order differential operators on $C^\infty({\mathbb R}^d)$, which depend on $x_0\in \mathbb R^d$, given by the formula
\begin{equation}\label{e:Ht}  
\mathcal H_t=\sum_{j=1}^{d} \nabla^2_{t,e_j}+ V(x_0+tZ),
\end{equation}
where $(e_1,\ldots, e_d)$ is the standard basis in $\mathbb R^d$ and $\nabla_t$ is the rescaled covariant derivative given by
\begin{equation}\label{e:nablat}
\nabla_{t,e_j}=\frac{1}{t} S^{-1}_t \left(\frac{\hbar}{i}\frac{\partial}{\partial x_j}-A_j(x)\right) S_t= \frac{1}{i}\frac{\partial}{\partial Z_j}-\frac{1}{t} A_j(x_0+tZ).
\end{equation}

By \eqref{e:ABP73a}, the operators $\nabla_{t,e_i}$ depend smoothly on $t=\sqrt{\hbar}$ up to $t=0$, 
\begin{equation}\label{e:nabla-exp}
\nabla_{t,e_j}=\nabla^{(x_0)}_{e_j}+O(t),
\end{equation}
where $\nabla^{(x_0)}$ is the covariant derivative associated with the magnetic potential $(A_{j,x_0}(Z))_{1\leq j\leq d}$:
\begin{equation*} 
\nabla^{(x_0)}_{e_j}=\frac{1}{i}\frac{\partial}{\partial Z_j}-A_{j,x_0}(Z), \quad Z\in \mathbb R^d, \quad j=1,\ldots,d.
\end{equation*}

By \eqref{e:ABP73}, we have an asymptotic expansion
\[
\frac{1}{t} A_j(x_0+tZ)=A_{j,x_0}(Z)+\sum_{r=1}^\infty t^{r}A_{j,r,x_0}(Z),
\]
where
\[
A_{j,r,x_0}(Z)=-\frac{1}{r+2}\sum_{|\alpha|=r}\sum_{k=1}^d(\partial^\alpha B_{jk})_{x_0} \,Z_k\frac{Z^\alpha}{\alpha!}
\]
is homogeneous of degree $r+1$.

Expanding the coefficients of the operator $\mathcal H_t$ in Taylor series in $t$, for any $m\in {\mathbb N}$, we get
\begin{equation}\label{e:Ht-formal}
\mathcal H_t=\mathcal H^{(0)}+\sum_{r=1}^m \mathcal H^{(r)}t^r+\mathcal O(t^{m+1}), 
\end{equation}
where there exists $m^\prime\in {\mathbb N}$ so that for every $k\in{\mathbb N}$ and $t\in [0,1]$ the derivatives up to order $k$ of the coefficients of the operator $\mathcal O(t^{m+1})$ are bounded by $Ct^{m+1}(1+|Z|)^{m^\prime}$. 

By \eqref{e:nabla-exp}, the leading term $\mathcal H^{(0)}$ coincides with the operator $\mathcal H^{(x_0)}$ given by \eqref{e:DeltaL0p}:
\[  
\mathcal H^{(0)}=-\sum_{j=1}^{d} (\nabla^{(x_0)}_{e_j})^2+V(x_0)=\mathcal H^{(x_0)}. 
\]
The next terms $\mathcal H^{(r)}, r\geq 1,$ are first order differential operators of the form
\begin{equation}\label{e:Hj}
\mathcal H^{(r)}=\sum_{j=1}^{d}b_{jr}\frac{\partial}{\partial Z_j}+c_r,
\end{equation}
where $$b_{jr}=2A_{j,r,x_0}$$ is a polynomial in $Z$ of degree $r+1$ and $$c_r=\sum_{j=1}^{d}\left(\frac{\partial A_{j,r+1,x_0}}{\partial Z_j} +\sum_{l=1}^{r} A_{j,l,x_0} A_{j,r-l,x_0}\right)+\sum_{|\alpha|=r}\partial^\alpha V(x_0)\frac{Z^\alpha}{\alpha!}$$ is a polynomial in $Z$ of degree $r+2$.  

By \eqref{e:H-equiv} and \eqref{scaling}, we have
\[
\varphi(H_\hbar/\hbar)=e^{i\Phi^{(x_0)}/\hbar} \varphi(H^{(x_0)}_\hbar/\hbar) e^{-i\Phi^{(x_0)}/\hbar}=e^{i\Phi^{(x_0)}/\hbar} S_t \varphi(\mathcal H_t)S^{-1}_t e^{-i\Phi^{(x_0)}/\hbar}.
\]
Accordingly, for the Schwartz kernels, we get
\begin{multline}\label{e:kernel-rel}
K_{\varphi(H_\hbar/\hbar)}(x_0+Z,x_0+Z^\prime)\\ =t^{-d}e^{i(\Phi^{(x_0)}(x_0+Z)-\Phi^{(x_0)}(x_0+Z^\prime))/\hbar} K_{\varphi(\mathcal H_{t,x_0})}\left(t^{-1}Z,t^{-1}Z^\prime\right), \quad Z, Z^\prime \in \mathbb R^{d}. 
\end{multline}

Thus, we have arrived at the family of $t$-dependent smoothing operators $\varphi(\mathcal H_{t,x_0})$ in $\mathbb R^d$ parametrized by $x_0\in \mathbb R^d$. We will study the asymptotic behavior of their Schwartz kernels $K_{\varphi(\mathcal H_{t,x_0})}\left(Z,Z^\prime\right)$ as $t\to 0$.  

\subsection{Norm estimates}\label{norm}
Next, we establish norm estimates for the operators $\varphi(\mathcal H_{t})$ and its derivatives of any order with respect to $t$. 

For $t>0$, set 
\[
\|u\|^2_{t,0}=\|u\|^2_{0}=\int_{\mathbb R^{d}}|u(Z)|^2dZ, \quad u\in C^\infty_c(\mathbb R^{d}),
\]
and, for any $m\in \mathbb N$ and $t>0$,
\[
\|u\|^2_{t,m}=\sum_{\ell=0}^m\sum_{j_1,\ldots,j_\ell=1}^{d}\|\nabla_{t,e_{j_1}}\cdots \nabla_{t,e_{j_\ell}}u\|^2_{t,0},\quad u\in C^\infty_c(\mathbb R^{d}), 
\]
where $\nabla_t$ is the rescaled connection defined by \eqref{e:nablat}. 

Let $\langle\cdot,\cdot\rangle_{t,m}$ denote the inner product on $C^\infty_c(\mathbb R^{d})$ corresponding to $\|\cdot\|^2_{t,m}$. Let $H^m_t$ be the Sobolev space of order $m$ with norm $\|\cdot\|_{t,m}$. For any integer $m<0$, we define the Sobolev space $H^{m}_t$ by duality.  

For a fixed $t>0$, the norm $\|\cdot\|_{t,m}$ is equivalent to the standard Sobolev norm given for $m\in \mathbb N$ by 
\[
\|u\|^2_{H^m(\mathbb R^{d})}=\sum_{|\alpha|\leq m}\|\partial^\alpha u\|^2_{0},\quad u\in C^\infty_c(\mathbb R^{d}).
\]
Therefore, the space $H^{m}_t$ coincides with the usual Sobolev space
$H^m(\mathbb R^{d})$ as a topological vector space. But this norm equivalence is not uniform as $t\to 0$. In order to have control of the Sobolev norms by the norms $\|\cdot\|_{t,m}$, uniform up to $t=0$, we introduce some weighted Sobolev norms with power weights as in \cite{ma-ma08}.

For $\alpha\in {\mathbb Z}^d_+$, we will use the standard notation
\[
Z^\alpha=Z_1^{\alpha_1}Z_2^{\alpha_2}\ldots Z_d^{\alpha_d}, \quad Z\in {\mathbb R}^d. 
\]
For any $t>0$, $m\in \mathbb Z$, and $M\in {\mathbb Z}_+$, we set
\[
\|u\|_{t,m,M}:= \sum_{|\alpha|\leq M}\left\|Z^{\alpha} u\right\|_{t,m}, \quad u\in C^\infty_c(\mathbb R^{d}).
\]
One can show that, for any $m\in {\mathbb N}$, there exists $C>0$ such that 
\begin{equation*}%\label{e:Hm-Hmm}
\|u\|_{H^m(\mathbb R^{d})} \leq C \|u\|_{t,m,m},\quad t\in (0,1], \quad u\in C^\infty(\mathbb R^{d}).
\end{equation*}

Moreover, as in \cite{Kor18}, we introduce an additional family of weight functions parameterized by $W\in \mathbb R^d$. Unlike \cite{Kor18}, where similar weight functions are exponential, here we consider power weights. 

For any $t>0$, $m\in \mathbb Z$, $M,N\in {\mathbb Z}_+$ and $W\in {\mathbb R}^d$, we set
\[
\|u\|_{t,m,M,N,W}= \sum_{|\alpha|\leq M,|\beta|\leq N}\left\|Z^{\alpha}(Z-W)^\beta u\right\|_{t,m}, \quad u\in C^\infty_c(\mathbb R^{d}).
\]
It is clear that
\[
\|u\|_{t,m,M,0,W}=\|u\|_{t,m,M}.
\]

\begin{theorem}\label{est-rem}
For any $r\geq 0$, $m, m^\prime\in {\mathbb Z}$, and $M,N\in \mathbb Z_+$, there exists $C>0$ such that, for any $t\in (0,t_0]$ and $W\in {\mathbb R}^d$, 
\[
\left\|\frac{\partial^r}{\partial t^r}\varphi(\mathcal H_{t}) u\right\|_{t,m,M,N,W}\leq C\|u\|_{t,m^\prime,M+2r,N,W}, \quad u\in C^\infty_c(\mathbb R^{d}).
\] 
\end{theorem}

The proof is based on the Helffer-Sjostrand formula \cite{HS-LNP345}:  
\begin{equation}\label{e:HS}
\varphi(\mathcal H_{t})=-\frac{1}{\pi }
\int_{\mathbb C} \frac{\partial \tilde{\varphi}}{\partial \bar \lambda}(\lambda)(\lambda-\mathcal H_{t})^{-1}d\mu d\nu,
\end{equation}
where $\tilde{\varphi}\in C^\infty_c({\mathbb C})$ is an almost-analytic extension of $\varphi$ satisfying 
\[
\frac{\partial \tilde{\varphi}}{\partial \bar \lambda}(\lambda)=O(|\nu|^\ell),\quad \lambda=\mu+i\nu, \quad \nu\to 0,
\]
for any $\ell\in {\mathbb N}$.

Applying the formula \eqref{e:HS} to the function $\psi(\lambda)=\varphi(\lambda)(a-\lambda)^K$ with any $K\in {\mathbb N}$ and $a<0$, we get 
\begin{equation}\label{e:HS1}
\varphi(\mathcal H_{t})=-\frac{1}{\pi }
\int_{\mathbb C} \frac{\partial \tilde{\varphi}}{\partial \bar \lambda}(\lambda) (a-\lambda)^K (\lambda-\mathcal H_{t})^{-1}(a-\mathcal H_{t})^{-K}d\mu d\nu.
\end{equation}
Now we differentiate \eqref{e:HS1}: 
\begin{equation}\label{e:diff-phi}
\frac{\partial^r}{\partial t^r}\varphi(\mathcal H_{t})=-\frac{1}{\pi }
\int_{\mathbb C} \frac{\partial \tilde{\varphi}}{\partial \bar \lambda}(a-\lambda)^K \frac{\partial^r}{\partial t^r}\left[(\lambda-\mathcal H_{t})^{-1}(a-\mathcal H_{t})^{-K}\right] d\mu d\nu,
\end{equation}
and use the following formula for the derivatives of the resolvent:
\begin{multline}\label{diff}
\frac{\partial^r}{\partial t^r}(\lambda-\mathcal H_{t})^{-1}\\ =\sum_{{\mathbf r}\in I_{r}} \frac{r!}{{\mathbf r}!} (\lambda - \mathcal H_{t})^{-1} \frac{\partial^{r_1}\mathcal H_{t}}{\partial t^{r_1}}(\lambda - \mathcal H_{t})^{-1}\cdots  \frac{\partial^{r_j}\mathcal H_{t}}{\partial t^{r_j}}(\lambda - \mathcal H_{t})^{-1},
\end{multline}
where
\[
I_{r}=\left\{ {\mathbf r}=(r_1,\ldots,r_j) : \sum_{i=1}^jr_i=r, r_i\in \mathbb Z_+ \right\}. 
\]
In the proof of Theorem \ref{est-rem}, we first prove weighted estimate for the resolvent $(\lambda-\mathcal H_{t})^{-1}$ and then apply the above formulas   (see \cite[Section 4]{bochner-trace} for more details).

\subsection{End of the proof of Theorem~\ref{t:main}}

Using the Sobolev embedding theorem, we derive pointwise estimates of the Schwartz kernels of smoothing operators from their mapping properties in Sobolev spaces in a standard way.  
 
\begin{theorem}\label{est-rem-pointwise}
For any $r\geq 0$ and $m\in \mathbb N$, there exists $M>0$ such that, for $N\in {\mathbb N}$, there exist $C>0$ such that, for any $t\in (0,1]$ and $Z,Z^\prime\in {\mathbb R}^{d}$, 
\[
\sup_{|\alpha|+|\alpha^\prime|\leq m}\Bigg|\frac{\partial^{|\alpha|+|\alpha^\prime|}}{\partial Z^\alpha\partial Z^{\prime\alpha^\prime}}\frac{\partial^r}{\partial t^r}K_{\varphi(\mathcal H_{t}) }(Z,Z^\prime)\Bigg| \leq C(1+|Z|+|Z^\prime|)^{M}(1+|Z-Z^\prime|)^{-N}. 
\]
\end{theorem}

Observe the following fact, which is an immediate consequence of the mean-value theorem: If $f$ is a differentiable function on the interval $(0,1)$ with values in a Banach space $B$ such that $\sup_{t\in (0,1)} \|f^\prime(t)\|_B<\infty$, then there exists $\lim_{t\to 0}f(t)$. From this fact and Theorem~\ref{est-rem-pointwise}, we infer that, for any $Z,Z^\prime\in {\mathbb R}^{d}$ and $r\geq 0$, there exists the limit 
\begin{equation}\label{e:defFr}
\lim_{t\to 0}\frac{1}{r!} \frac{\partial^r}{\partial t^r}K_{\varphi(\mathcal H_{t})}(Z,Z^\prime) =F_{r}(Z,Z^\prime),
\end{equation}
for some $F_{r}=F_{r,x_0}\in C^\infty(\mathbb R^{d}\times \mathbb R^{d})$.
Consider an operator $F_r : C^\infty_c(\mathbb R^{d})\to C^\infty(\mathbb R^{d})$ defined by the Schwartz kernel $F_r(Z,Z^\prime)$. 
Using the Taylor formula 
\[
\varphi(\mathcal H_{t})-\sum_{r=0}^j{ F_{r} t^r}=\frac{1}{j!}\int_0^t(t-\tau)^j\frac{\partial^{j+1}\varphi(\mathcal H_{t})}{\partial t^{j+1}}(\tau) d\tau, \quad t\in [0,1],  
\]
and Theorem \ref{est-rem-pointwise}, we immediately get the following theorem.

 \begin{theorem}\label{t:thm7.2}
For any $j,m\in \mathbb N$, there exists $M>0$ such that, for any $N\in {\mathbb N}$,  there exists $C>0$ such that for any $t\in (0,1]$ and $Z,Z^\prime\in {\mathbb R}^{d}$, 
\begin{multline*}
\sup_{|\alpha|+|\alpha^\prime|\leq m}\Bigg|\frac{\partial^{|\alpha|+|\alpha^\prime|}}{\partial Z^\alpha\partial Z^{\prime\alpha^\prime}}\Bigg(K_{\varphi(\mathcal H_{t})}(Z,Z^\prime)
-\sum_{r=0}^jF_{r}(Z,Z^\prime)t^r\Bigg)\Bigg|  \\
\leq Ct^{j+1}(1+|Z|+|Z^\prime|)^{M}(1+|Z-Z^\prime|)^{-N}. 
\end{multline*}
\end{theorem}

By \eqref{e:kernel-rel}, this completes the proof of the asymptotic expansion \eqref{e:main} in Theorem~\ref{t:main}.

\subsection{Computation of the coefficients}\label{s:computation}
Since $\lim_{t\to 0}\frac{\partial^{j}\mathcal H_{t}}{\partial t^{j}}=j!\mathcal H^{(j)}$, { where $\mathcal H^{(j)}$ is the coefficient in the asymptotic expansion \eqref{e:Ht-formal} given by \eqref{e:Hj},} from \eqref{diff}, we get
\begin{multline}\label{diff0}
\lim_{t\to 0}\frac{\partial^r}{\partial t^r}(\lambda-\mathcal H_{t})^{-1}\\ =\sum_{{\mathbf r}\in I_{r}} {r}!(\lambda - \mathcal H^{(0)})^{-1} \mathcal H^{(r_1)}(\lambda - \mathcal H^{(0)})^{-1}\cdots  \mathcal H^{(r_j)}(\lambda - \mathcal H^{(0)})^{-1}.
\end{multline}
By \eqref{e:diff-phi} with $K=0$ and \eqref{e:defFr}, we infer that
\begin{multline}\label{e:Fr-phi}
F_r= -\frac{1}{\pi }\sum_{{\mathbf r}\in I_{r}}{ \int_{\mathbb C}} \frac{\partial \tilde{\varphi}}{\partial \bar \lambda}(\lambda) (\lambda - \mathcal H^{(0)})^{-1}\\ \times \mathcal H^{(r_1)}(\lambda - \mathcal H^{(0)})^{-1}\cdots  \mathcal H^{(r_j)}(\lambda - \mathcal H^{(0)})^{-1} d\mu d\nu.
\end{multline}

For $r=0$, we immediately get
\[
F_0= -\frac{1}{\pi }\int_{\mathbb C} \frac{\partial \tilde{\varphi}}{\partial \bar \lambda}(\lambda) (\lambda - \mathcal H^{(0)})^{-1} d\mu d\nu=\varphi(\mathcal H^{(0)}),
\]
which proves Theorem~\ref{t:leading-coefficient}. 

To compute lower order coefficients and prove the formula \eqref{e:Fr}, we use \eqref{e:Fr-phi} and the technique of creation and annihilation operators (see \cite{bochner-trace} for more details). 

\section{Localization of the spectral projection and eigenfunctions} 
 
\subsection{Discreteness of the spectrum} 
The following proposition plays a crucial role both in the proof of Theorem \ref{t:ess-spectrum} and in the proof of Theorem \ref{t:eigenest}.
 
  \begin{proposition}\label{p:estimate}
 Let $[a,b]\subset \mathbb R$ be a bounded interval and 
 $$
{ \Omega_{[a,b]}:=\mathbb R^{2n}\setminus \mathcal K_{[a,b]}=\{x\in \mathbb R^{2n} : \Sigma_x\cap [a,b]=\emptyset\}.}
 $$ 
 Then there exist $C>0$ and $\hbar_0>0$ such that, for any $\lambda\in (a,b)$, for any $\hbar\in (0,\hbar_0]$ and for any $u\in C^\infty({\mathbb R}^{2n})$ compactly supported in $\Omega_{[a,b]}$, we have 
 \[
 \|(H_{\hbar}/\hbar-\lambda)u\|\geq (d(\lambda,\Sigma)-C\hbar^{1/4})\|u\|,
 \]
 where $d(\lambda,\Sigma)$ stands for the distance from $\lambda$ to $\Sigma$. 
 \end{proposition}
   
The proof of Proposition \ref{p:estimate} follows the lines of the proof given in \cite[pp. 106--110]{HM88}. It is based on approximation of the operator $H_{\hbar}$ by the model operator $\mathcal H^{(x_0)}_{\hbar}$ in a sufficiently small neighborhood of an arbitrary point $x_0$ (with the use of a special gauge given by local Darboux coordinates for the magnetic field $B$) and localization with the help of standard partition of unity of size $\hbar^{1/4}$. 
  
 \begin{remark}
 In case $ \Omega_{[a,b]}={\mathbb R}^{2n}$, Proposition \ref{p:estimate} immediately implies  Theorem~\ref{t:spectrum}. This proof of Theorem~\ref{t:spectrum} is different from the original proof given in \cite{higherLL}, which is based on a construction of an approximate inverse for the operator $H_{\hbar}/\hbar-\lambda$ (see also \cite{charles21,FT}). 
 \end{remark}
 
Theorem \ref{t:ess-spectrum} follows immediately from Proposition \ref{p:estimate} and the following lemma, which is a consequence of \cite[Lemma 2.1]{M88}.
 
 \begin{lemma}
Let $\lambda\in\mathbb R$ and $\hbar>0$. Suppose that there exist  $\delta>0$ and a compact $K\subset {\mathbb R}^{2n}$ such that
\begin{equation}
 \|(H_{\hbar}-\lambda)u\|\geq \delta\|u\|
 \end{equation}
 for any $u\in H^2({\mathbb R}^{2n})$ supported in ${\mathbb R}^{2n}\setminus K$. Then $\lambda\not\in \sigma_{\rm ess}(H_{\hbar}).$
 \end{lemma}

\subsection{Eigenfunction estimates}
 Asymptotic localization of eigenfunctions of the magnetic Schr\"odinger operator associated with eigenvalues below the bottom of the essential spectrum usually follows from Agmon type estimates. But when we consider eigenvalues in gaps of the essential spectrum, this method doesn't work, because it is based on the use of quadratic forms. Instead, we use a slight modification of the method developed in \cite{higherLL,bochner-trace} (see also the references therein) to prove exponential localization away the diagonal for Schwartz kernels of functions of the Bochner-Schr\"odinger operator and in \cite{FLRV24} to prove exponential localization of the boundary states of the Robin magnetic  Laplacian away the boundary. The proof of Theorem \ref{t:eigenest} follows the scheme of the proof of Proposition 2.1 in \cite{FLRV24}. The main difference is that, instead of weighted estimates for the resolvent used in \cite{FLRV24}, our proof is based on the norm estimates given by Proposition~\ref{p:estimate}.

We will use some weight functions. 
By standard averaging, one can show that, for any $\hbar>0$, there exists a function $\Phi_{\hbar}\in C^\infty({\mathbb R}^{2n})$, satisfying the following conditions:

(1) we have
\begin{equation}\label{(1.1)}
\vert \Phi_{\hbar}(x) - d (x,\mathcal K_{[a,b]}\vert  < \hbar^{1/2}\;,
\quad x\in {\mathbb R}^{2n}, \quad \hbar>0;
\end{equation}

(2) for any $k>0$, there exists $C_k>0$ such that
\begin{equation}\label{e:chi-p}
\hbar^{(k-1)/2}\left|\nabla^k\Phi_{\hbar}(x)\right|<C_k, \quad x\in {\mathbb R}^{2n}, \quad \hbar>0. 
\end{equation} 
 
Define a family of differential operators on $C^\infty({\mathbb R}^{2n})$ by
\begin{align}\label{e:weight-operator}
H_{\hbar,\tau}:= e^{\tau\Phi_{\hbar}/\hbar^{1/2}} H_{\hbar} e^{-\tau\Phi_{\hbar}/\hbar^{1/2}},\quad \hbar>0, \quad \tau\in {\mathbb R}.
\end{align}
An easy computation gives 
\begin{equation}\label{e:DpaW}
H_{\hbar,\tau}=H_{\hbar}+\tau A_{\hbar}+\tau^2B_{\hbar},
\end{equation}
where 
\begin{equation} \label{e:ApBp}
A_{\hbar}=-2i \hbar^{1/2} d \Phi_{\hbar}\cdot \nabla_{\hbar}+\hbar^{3/2}\Delta \Phi_{\hbar}, \quad B_{\hbar}=-\hbar |d \Phi_{\hbar}|^2.
\end{equation}
Here, for $u \in C^\infty({\mathbb R}^{2n})$, $d\Phi_{\hbar}\cdot\nabla_{\hbar}u \in C^\infty({\mathbb R}^{2n})$ stands for the pointwise inner product of $d\Phi_{\hbar}\in C^\infty({\mathbb R}^{2n},T^*{\mathbb R}^{2n})$ and $\nabla_{\hbar}u\in C^\infty({\mathbb R}^{2n},T^*{\mathbb R}^{2n})$ determined by the Riemannian metric.

\begin{proof}[Proof of Theorem \ref{t:eigenest}]%\label{s:eigenest}
Suppose that $u_{\hbar}\in C^\infty({\mathbb R}^{2n})\cap L^2({\mathbb R}^{2n})$ is such that 
\[
H_{\hbar}u_{\hbar}=\lambda_{\hbar}u_{\hbar}
\]
with some $\hbar>0$ and $\lambda_{\hbar}\in (a,b)$. 
Then, for $v_{\hbar}=e^{\tau\Phi_{\hbar}/\hbar^{1/2}}u_{\hbar}$, we have 
\begin{equation}\label{e:Hpa-lambda}
H_{\hbar,\tau}v_{\hbar}=\lambda_{\hbar}v_{\hbar}.
\end{equation}

Choose an arbitrary $a_2$ and $b_2$ such that $a_1>a_2>a$ and $b_1<b_2<b$. Let 
\[
\Omega_1=\Omega_{[a_2,b_2]}=\{x\in {\mathbb R}^{2n} : \Sigma_x\cap [a_2, b_2]=\emptyset\}.
\]
This is an open subset of ${\mathbb R}^{2n}$, which contains $\bar\Omega=\overline{\Omega_{[a,b]}}$. Moreover, since $B$ and $V$ are $C^\infty$-bounded, there exists $\epsilon>0$ such that  $\Omega_{1,\epsilon}:=\{x\in \Omega_1 : d(x,\partial \Omega_1)>\epsilon\}$ contains $\bar\Omega$.

Now let $\phi_\hbar \in C^\infty_b({\mathbb R}^{2n})$ be supported in $\Omega_1$ and $\phi_\hbar\equiv 1$ on $\Omega_{1, \hbar^{1/2}}=\{x\in \Omega : d(x,\partial \Omega)>\hbar^{1/2}\}\subset \Omega$:
\[
|\nabla\phi_\hbar|<\hbar^{-1/2}, \quad |\nabla^2\phi_\hbar|<\hbar^{-1}.
\] 
Then
\[
H_{\hbar,\tau}(\phi_\hbar v_{\hbar})=\lambda_{\hbar}\phi_\hbar v_{\hbar}+[H_{\hbar,\tau}, \phi_\hbar] v_{\hbar}.
\]
By Proposition~\ref{p:estimate}, 
there exist $C_0>0$ and $\hbar_0>0$, such that for any $\hbar\in (0,\hbar_0]$, we have 
\begin{equation}\label{e:lower-est}
\|(H_{\hbar}-\lambda_{\hbar})\phi_\hbar v_{\hbar}\|\geq  C_0\hbar\|\phi_\hbar v_{\hbar}\|.
\end{equation}

\begin{lemma}\label{l:lower-est-a}
For any $\tau>0$ small enough,  there exists $C_0>0$, such that for any $\hbar\in (0,\hbar_0]$, we have
\begin{equation}\label{e:lower-est-a}
\|(H_{\hbar,\tau}-\lambda_{\hbar})\phi_\hbar v_{\hbar}\|\geq  C_0\hbar\|\phi_\hbar v_{\hbar}\|.
\end{equation}
\end{lemma}

\begin{proof}
We can write
\begin{equation}\label{e:Hpa-lp}
\|(H_{\hbar,\tau}-\lambda_{\hbar})\phi_\hbar v_{\hbar}\|\geq \|(H_{\hbar}-\lambda_{\hbar})\phi_\hbar v_{\hbar}\|-\|(H_{\hbar,\tau}-H_{\hbar})\phi_\hbar v_{\hbar}\|.
\end{equation}
By \eqref{e:DpaW} and \eqref{e:ApBp}, we have
\[
H_{\hbar,\tau}-H_{\hbar}=\tau(-2i \hbar^{1/2} d \Phi_{\hbar}\cdot \nabla_{\hbar}+\hbar^{3/2}\Delta \Phi_{\hbar})-\tau^2\hbar |d \Phi_{\hbar}|^2.
\]
Using \eqref{e:chi-p}, we get 
\begin{multline}\label{e:Hpa-Hp}
\|(H_{\hbar,\tau}-H_{\hbar})\phi_\hbar v_{\hbar} \|\\ \leq C_1 \tau \hbar^{1/2} \|\nabla_{\hbar} \phi_\hbar v_{\hbar}\| +C_2\tau\hbar\|\phi_\hbar v_{\hbar} \|+C_3\tau^2\hbar\|\phi_\hbar v_{\hbar} \|.
\end{multline}

To estimate the first term in the right hand side of \eqref{e:Hpa-Hp}, we proceed as follows:
\begin{align*}
\|\nabla_{\hbar}(\phi_\hbar v_{\hbar})\|^2=& ((\nabla_{\hbar})^2v_{\hbar}, v_{\hbar})=((H_{\hbar}-\hbar V)\phi_\hbar v_{\hbar}, \phi_\hbar v_{\hbar})
\\
= & ((H_{\hbar}-\lambda_{\hbar})\phi_\hbar v_{\hbar}, \phi_\hbar v_{\hbar})+((\lambda_{\hbar}-\hbar V)\phi_\hbar v_{\hbar}, \phi_\hbar v_{\hbar})
\\
\leq & \|(H_{\hbar}-\lambda_{\hbar})\phi_\hbar v_{\hbar}\|\|\phi_\hbar v_{\hbar}\|+C\hbar \|\phi_\hbar v_{\hbar}\|^2
\\
\leq & \epsilon^2 \hbar^{-1}\|(H_{\hbar}-\lambda_{\hbar})\phi_\hbar v_{\hbar}\|^2+(C+\epsilon^{-2})\hbar\|\phi_\hbar v_{\hbar}\|^2
\end{align*}
with an arbitrary $\epsilon>0$ to be chosen later. Eventually, we get
\[
\|\nabla_{\hbar}(\phi_\hbar v_{\hbar})\|\leq \epsilon \hbar^{-1/2}\|(H_{\hbar}-\lambda_{\hbar})\phi_\hbar v_{\hbar}\|+(C+\epsilon^{-1}) \hbar^{1/2} \|\phi_\hbar v_{\hbar}\|. 
\]
Using this estimate, from \eqref{e:Hpa-Hp}, we get
\begin{multline*}
\|(H_{\hbar,\tau}-H_{\hbar})\phi_\hbar v_{\hbar} \|\leq C_1\tau\epsilon \|(H_{\hbar}-\lambda_{\hbar})\phi_\hbar v_{\hbar}\| \\ +(C_2+C_3\epsilon^{-1})\tau \hbar \|\phi_\hbar v_{\hbar}\| +C_4 \tau^2\hbar\|\phi_\hbar v_{\hbar} \|.
\end{multline*}
Choosing $\epsilon$ such that $C_1\tau\epsilon=\frac 12$, we get
\[
\|(H_{\hbar,\tau}-H_{\hbar})\phi_\hbar v_{\hbar} \|\leq \frac 12 \|(H_{\hbar}-\lambda_{\hbar})\phi_\hbar v_{\hbar}\|+C_5\hbar \tau\|\phi_\hbar v_{\hbar} \|+C_6\hbar \tau^2\|\phi_\hbar v_{\hbar} \|.
\]
and, using \eqref{e:lower-est}, from \eqref{e:Hpa-lp}, we get
\begin{multline*}
\|(H_{\hbar,\tau}-\lambda_{\hbar})\phi_\hbar v_{\hbar}\|\geq \frac 12 \|(H_{\hbar}-\lambda_{\hbar})\phi_\hbar v_{\hbar}\|-C_5\hbar \tau\|\phi_\hbar v_{\hbar} \|-C_6\hbar \tau^2\|\phi_\hbar v_{\hbar} \| \\
\geq \frac 12 C_0\hbar \|\phi_\hbar v_{\hbar}\|-C_5\hbar \tau\|\phi_\hbar v_{\hbar} \|-C_6\hbar \tau^2\|\phi_\hbar v_{\hbar} \|.
\end{multline*}
Taking $\tau$ small enough, we complete the proof.
\end{proof}

Let us fix $\tau$ as in Lemma \ref{l:lower-est-a}. By \eqref{e:Hpa-lambda}, we have
\begin{equation}\label{e:H=comm}
(H_{\hbar,\tau}-\lambda_{\hbar})\phi_\hbar v_{\hbar}=[H_{\hbar,\tau},\phi_\hbar]v_{\hbar}.
\end{equation}
By \eqref{e:DpaW} and \eqref{e:ApBp}, we compute
\begin{equation}\label{e:commHp-phi}
[H_{\hbar,\tau}, \phi_\hbar]=(-2i\hbar d\phi_\hbar\cdot \nabla_\hbar+\hbar^2\Delta\phi_\hbar)-2\tau \hbar^{3/2}d \Phi_{\hbar}\cdot d\phi_\hbar.
\end{equation}

Since  $d\phi_\hbar$ and $\Delta\phi_\hbar$ are supported in $\Omega_1\setminus \overline{\Omega_{1,\hbar^{1/2}}}$, we have
\begin{equation}\label{e:est-psi}
\|(d \Phi_{\hbar}\cdot d \phi_\hbar)v_{\hbar}\|\leq C\hbar^{-1/2}\|\psi_\hbar v_{\hbar}\|, \quad \|\Delta\phi_\hbar\, v_{\hbar}\|\leq C\hbar^{-1}\| \psi_\hbar v_{\hbar}\|,
\end{equation}
where $\psi_\hbar \in C^\infty_b({\mathbb R}^{2n})$ is supported in ${\mathbb R}^{2n} \setminus \overline{\Omega_{2\hbar^{1/2}}}$ and $\psi_\hbar\equiv 1$ on ${\mathbb R}^{2n}\setminus \Omega_{\hbar^{1/2}}$, in particular on ${\rm supp}\,d\phi_\hbar\subset \Omega \setminus \overline{\Omega_{\hbar^{1/2}}}$.

Therefore, by \eqref{e:commHp-phi}, \eqref{e:chi-p} and \eqref{e:est-psi}, we get 
\[
\|[H_{\hbar,\tau},\phi_\hbar]v_{\hbar}\|\leq 2\hbar\|d\phi_\hbar\cdot \nabla_\hbar v_{\hbar}\|+C_2\hbar\|\psi_\hbar v_{\hbar}\|.
\]

Next, we show the following estimate:
\[
\|d\phi_\hbar\cdot\nabla_\hbar v_{\hbar}\|\leq C\|\psi_\hbar v_{\hbar}\|. 
\]
Its proof is standard and quite lengthy, so we will omit it. We infer that 
\[
\|[H_{\hbar,\tau}, \phi_\hbar]v_{\hbar}\|\leq C_1\hbar\|\psi_\hbar v_{\hbar}\|.
\]
From this estimate, taking into account \eqref{e:lower-est} and \eqref{e:H=comm}, we get 
\[
\|\phi_\hbar v_{\hbar}\|\leq C_1\|\psi_\hbar v_{\hbar}\|.
\]
Now we proceed as follows: 
\begin{multline*}
\int_{\Omega} e^{2\tau\Phi_{\hbar}(x)/\hbar^{1/2}}|u_{\hbar}(x)|^2dx\leq 
\int_{\Omega_{1,\hbar^{1/2}}} e^{2\tau\Phi_{\hbar}(x)/\hbar^{1/2}}|u_{\hbar}(x)|^2dx\\
=\|v_{\hbar}\|^2_{L^2(\Omega_{\hbar^{1/2}})} \leq \|\phi_\hbar v_{\hbar}\|^2 \leq C^2_1 \|\psi_\hbar v_{\hbar}\|^2 = C^2_1 \|\psi_\hbar v_{\hbar}\|^2_{L^2({\mathbb R}^{2n} \setminus \overline{\Omega_{1,2\hbar^{1/2}}})}\\ \leq C^2_1 \|v_{\hbar}\|^2_{L^2({\mathbb R}^{2n} \setminus \overline{\Omega_{1, 2\hbar^{1/2}}})}=C^2_1 \int_{{\mathbb R}^{2n} \setminus \overline{\Omega_{1,2\hbar^{1/2}}}} e^{2\tau\Phi_{\hbar}(x)/\hbar^{1/2}}|u_{\hbar}(x)|^2dx\\ \leq C^2_1 \|u_{\hbar}\|^2, 
\end{multline*}
since $\Phi_{\hbar}=0$ on ${\mathbb R}^{2n} \setminus \overline{\Omega_{1,2\hbar^{1/2}}}\subset {\mathbb R}^{2n} \setminus\Omega $, that completes the proof of Theorem \ref{t:eigenest}. \end{proof}

 \noindent
{\bf Acknowledgment:} I am grateful to the anonymous referee for his/her valuable comments, which helped me to improve the paper.


\begin{thebibliography}{99}
\bibitem{ABP73}
M. Atiyah, R. Bott and V. K. Patodi, {\it On the heat equation and the index theorem}, Invent. Math. {\bf 19} (1973), 279--330. 

\bibitem{BL} 
J.-M. Bismut and G. Lebeau, {\it Complex immersions and Quillen metrics},
Inst. Hautes \'Etudes Sci.\ Publ.\ Math. {\bf 74} (1991). 
 
 \bibitem{BRV}
Y. G. Bonthonneau, N. Raymond and S. V\~{u} Ng\d{o}c, {\it Exponential localization in 2D pure magnetic wells}, Ark. Mat. {\bf 59} (2021), 53--85. 

\bibitem{charles21}
L. Charles, {\it On the spectrum of non degenerate magnetic Laplacian}, Anal. PDE {\bf 17} (2024),  1907--1952.

\bibitem{charles24}
L. Charles, {\it Landau levels on a compact manifold}, Ann. H. Lebesgue {\bf 7} (2024), 69--121. 

\bibitem{CN98}
H. D. Cornean and Gh. Nenciu, {\it  On eigenfunction decay of two dimensional magnetic
Schr\"odinger operators,} Commun. Math. Phys., {\bf 192} (1998), 671--685 

\bibitem{CN00}
H. D. Cornean and Gh. Nenciu, {\it Two dimensional magnetic Schr\"odinger operators: width
of mini-bands in the tight-binding approximation, } Ann. Henri Poincar\'e {\bf 1} (2000), 203--222

\bibitem{dai-liu-ma}
X.~Dai, K.~Liu and X.~Ma, 
{\it On the asymptotic expansion of {B}ergman
  kernel}, J.\ Differential Geom.\ {\bf 72} (2006), 1--41.

\bibitem{Demailly85}
J.-P. Demailly, {\it Champs magn\'{e}tiques et in\'{e}galit\'{e}s de {M}orse pour la {$d''$}-cohomologie}, Ann. Inst. Fourier (Grenoble) {\bf 35} (1985), 189--229.

\bibitem{Demailly91}
J.-P. Demailly, {\it Holomorphic Morse inequalities.} In {\it Several complex variables and complex geometry, Part 2 (Santa Cruz, CA, 1989)}, 93--114, Proc. Sympos. Pure Math., 52, Part 2, Amer. Math. Soc., Providence, RI, 1991.

\bibitem{FLRV24}
R. Fahs, L. Le Treust, N, Raymond and  S. Vu Ngoc, {\it Boundary states of the Robin magnetic Laplacian}, Doc. Math. {\bf 29} (2024), 1157--1200.

\bibitem{FT}
F. Faure and M. Tsujii, {\it Prequantum transfer operator for symplectic Anosov diffeomorphism.} Ast\'{e}risque No. {\bf 375} (2015).

\bibitem{HK14}
B. Helffer and Yu. A. Kordyukov, {\it Semiclassical spectral asymptotics for a magnetic {S}chr\"{o}dinger operator with non-vanishing magnetic field.} Geometric methods in physics, 259--278, Trends Math., Birkh\"auser/Springer, Cham, 2014.

\bibitem{HK15}
B. Helffer and Yu. A. Kordyukov, {\it Accurate semiclassical spectral asymptotics for a two-dimensional magnetic {S}chr\"{o}dinger operator.} Ann. Henri Poincar\'{e} {\bf 16} (2015), 1651--1688. 

\bibitem{HKRV16}
B. Helffer, Y. Kordyukov, N. Raymond, and S. V\~{u} Ng\d{o}c, {\it Magnetic wells in dimension three,} Anal. PDE {\bf 9} (2016), 1575--1608.

\bibitem{HM88}
B. Helffer and A. Mohamed, {\it Caract\'{e}risation du spectre essentiel de l'op\'{e}rateur de {S}chr\"{o}dinger avec un champ magn\'{e}tique}, Ann. Inst. Fourier (Grenoble) {\bf 38} (1988), 95--112.

\bibitem{HS-LNP345}
B. Helffer and J. Sj\"{o}strand, {\it \'{E}quation de {S}chr\"{o}dinger avec champ magn\'{e}tique et \'{e}quation de {H}arper.} In {\it Schr\"{o}dinger operators ({S}\o nderborg, 1988)}, 118--197, Lecture Notes in Phys., 345, Springer, Berlin, 1989.

\bibitem{HS15}
P. Hislop and E. Soccorsi, {\it Edge states induced by Iwatsuka Hamiltonians with positive magnetic fields,} J. Math. Anal. Appl. {\bf 422} (2015), 594--624.

 \bibitem{I85}
A. Iwatsuka, {\it Examples of absolutely continuous Schr\"{o}dinger operators in magnetic fields,} Publ. Res. Inst. Math. Sci. {\bf 21} (1985), 385--401.

\bibitem{Kor18}
Yu. A. Kordyukov, {\it On asymptotic expansions of generalized Bergman kernels on symplectic manifolds (Russian)}, Algebra i Analiz {\bf 30} (2018), no. 2, 163--187; translation in St. Petersburg Math. J. {\bf 30} (2019), no. 2, 267--283.

\bibitem{higherLL} Yu. A. Kordyukov, {\it Semiclassical spectral analysis of the Bochner-Schr\"odinger operator on symplectic manifolds of bounded geometry,} Anal. Math. Phys. {\bf 12} (2022), no. 1, Paper No. 22, 37 pp.

\bibitem{jst}
Yu. A. Kordyukov, 
{\it Berezin-Toeplitz quantization asssociated with higher Landau levels of the Bochner Laplacian,} J. Spectral Theory {\bf 12} (2022), 143--167.

\bibitem{bg-guant}
Yu. A. Kordyukov, {\it Berezin-Toeplitz quantization on symplectic 
manifolds of bounded geometry (Russian),} Mat. Zametki, {\bf 112} (2022), no 4, 586 -- 600; translation in Math. Notes {\bf 112} (2022), no. 4, 576 -- 587.

\bibitem{UMN-trace} Yu. A. Kordyukov, {\it Trace formula for the magnetic Laplacian at zero energy level,} Uspekhi Mat. Nauk {\bf 77} (2022), no. 6, 159--202; translation in Russian Math. Surv. {\bf 77} (2022), no. 6, 1107--1148. 

\bibitem{bochner-trace}  Yu. A. Kordyukov, {\it Semiclassical asymptotic expansions for functions of the Bochner-Schr\"odinger operator}, Russ. J. Math. Phys. {\bf 30} (2023), 192--208.

\bibitem{essential} Yu. A. Kordyukov, {\it Exponential localization for eigensections of the Bochner- 
Schr\"odinger operator,} Russ. J. Math. Phys. {\bf 31} (2024), 461--476.

\bibitem{ko-ma-ma} Yu. A. Kordyukov, X. Ma and G. Marinescu, 
{\it Generalized Bergman kernels on symplectic manifolds of bounded geometry,} Comm. Partial Differential Equations {\bf 44} (2019), 1037--1071.

\bibitem{ma-ma:book}
X. Ma and G.~Marinescu, {\it Holomorphic Morse inequalities and Bergman kernels}, Progress in Mathematics, 254. Birkh\"auser Verlag, Basel, 2007. 

\bibitem{ma-ma08} 
X. Ma and G.~Marinescu, {\it Generalized Bergman kernels on symplectic manifolds}, Adv. Math. {\bf 217} (2008), 1756--1815.

\bibitem{ma-ma15} X.~Ma and G.~Marinescu, 
{\it Exponential estimate for the asymptotics of Bergman kernels},  
Math. Ann. {\bf 362} (2015),  1327--1347.

\bibitem{M88}
A. Mohamed, {\it Quelques remarques sur le spectre de l'op\'{e}rateur de {S}chr\"{o}dinger avec un champ magn\'{e}tique}, Comm. Partial Differential Equations {\bf 13} (1988), 1415--1430.

\bibitem{M}
L. Morin,  {\it A semiclassical Birkhoff normal form for symplectic magnetic wells}, J. Spectr. Theory {\bf 12} (2022), 459--496.

\bibitem{M24}
L. Morin,  {\it A semiclassical Birkhoff normal form for constant-rank magnetic fields}, Anal. PDE {\bf 17} (2024), 1593--1632.

\bibitem{N02}
G. Nenciu, {\it  On asymptotic perturbation theory for quantum mechanics: almost invariant subspaces and gauge invariant magnetic perturbation theory. } J. Math. Phys. {\bf 43} (2002), 1273--1298.

\bibitem{RS23}
N. Raymond and E. Soccorsi, {\it Magnetic quantum currents in the presence of a Neumann wall,} J. Math. Phys. {\bf 64} (2023), Paper No. 073506, 18 pp.

\bibitem{RV}
N. Raymond and S. V\~{u} Ng\d{o}c, {\it Geometry and spectrum in 2D magnetic wells}, Ann. Inst. Fourier (Grenoble) {\bf 65} (2015), 137--169.  

\bibitem{Savale17}
N. Savale, {\it Koszul complexes, Birkhoff normal form and the magnetic Dirac operator}, Anal. PDE {\bf 10} (2017), 1793--1844.

\bibitem{Savale18}
N. Savale, {\it A Gutzwiller type trace formula for the magnetic Dirac operator}, Geom. Funct. Anal. {\bf 28} (2018), 1420--1486.
\end{thebibliography}
\end{document}